\documentclass[10pt]{amsart}
\usepackage{amsmath}
\usepackage{amsfonts}
\usepackage{amssymb}
\usepackage{graphicx}
\usepackage{amsthm,graphicx,color,yfonts}
\usepackage{pdfsync}
\usepackage{epstopdf}
\usepackage{epsfig}
\usepackage{marginnote} 
\usepackage{esint} 
\usepackage[colorlinks=true]{hyperref}
\hypersetup{linkcolor=red,citecolor=blue,filecolor=dullmagenta,urlcolor=darkblue} % coloured links

\numberwithin{equation}{section}

\theoremstyle{plain}
\newtheorem{theorem}{Theorem}[section]

\newtheorem{lemma}[theorem]{Lemma}
\newtheorem*{de-lemma}{Lemma}

\theoremstyle{remark}

\newtheorem*{conjecture}{Conjecture}
\theoremstyle{definition}

\DeclareMathOperator{\supp}{supp}

\newcommand{\dd}{\mathrm{d}}
\newcommand{\R}{\mathbb{R}}\newcommand{\C}{\mathbb{C}}
\newcommand{\SF}{\mathbb{S}}
\newcommand{\n}{\textbf{\em n}}
\newcommand{\ve}{\epsilon}

\providecommand{\MR}{\relax\ifhmode\unskip\space\fi MR }
\providecommand{\href}[2]{#2}

\begin{document}

\title[The connecting solution of the Painlev\'e model]{The connecting solution of the Painlev\'e phase transition model}

\author{Marcel G. Clerc}
\address{Departamento de F\'{i}sica, FCFM, Universidad de Chile, Casilla 487-3, Santiago,
Chile.}
\email{marcel@dfi.uchile.cl}
%\thanks{M.G. Clerc was partially supported by Fondecyt 1180903.}

\author{Micha{\l } Kowalczyk}
\address{Departamento de Ingenier\'{\i}a Matem\'atica and Centro
de Modelamiento Matem\'atico (UMI 2807 CNRS), Universidad de Chile, Casilla
170 Correo 3, Santiago, Chile.}
\email {kowalczy@dim.uchile.cl}
%\thanks{M. Kowalczyk was partially supported by Chilean research grants Fondecyt 1130126 and 1170164, Fondo Basal AFB170001 CMM-Chile}

\author{Panayotis Smyrnelis}
\address{Centro
de Modelamiento Matem\'atico (UMI 2807 CNRS), Universidad de Chile, Casilla
170 Correo 3, Santiago, Chile.}
\email{psmyrnelis@dim.uchile.cl}
%\thanks{P. Smyrnelis was partially supported by Fondo Basal CMM-Chile and Fondecyt postdoctoral grant 3160055}


%\date{\today}

\begin{abstract}
The  second Painlev\'e O.D.E. $y''-xy-2y^3=0$, $x\in \R,$ is known to play an important role in  the theory of integrable systems,  random matrices,  Bose-Einstein condensates and other problems.	
The  generalized second Painlev\'e equation $\Delta y -x_1 y - 2 y^3=0$, $(x_1,x_2)\in \R^2$, is  obtained by multiplying by $-x_1$ the linear term $u$ of the Allen-Cahn equation $\Delta u =u^3-u$. It involves a non autonomous potential $H(x_1,y)$ which is bistable for every fixed $x_1<0$, and thus describes as the Allen-Cahn equation a phase transition model. 
The scope of this paper is to construct a solution $y$ connecting along the vertical direction $x_2$, the two branches of minima of $H$ parametrized by $x_1$. This solution plays a similar role that the heteroclinic orbit for the Allen-Cahn equation. It is the the first to our knowledge solution of the Painlev\'e P.D.E. both relevant from the applications point of view (liquid crystals), and mathematically interesting.
\end{abstract}

\maketitle

\section{The Allen-Cahn and Painlev\'e phase transition models}
A standard phase transition model is given by the Allen-Cahn O.D.E.:
\begin{equation}\label{acb1}
	u''=u^3-u, \qquad \mbox{in}\ \R,
\end{equation}
that can alternatively be written $u''=W'(u)$, where $W(u)=\frac{1}{4}(u^2-1)^2$ is a double well potential. In this model, $u$ describes the mass fraction of the two phases of a substance (e.g. an alloy), and takes values approximately $+1$ or $-1$ for the pure phases. Equation \eqref{acb1} has variational structure. Let  
\[
E_{\mathrm{AC}}(u,(a,b)):=\int_a^b \Big(\frac{1}{2}|u'|^2+\frac{1}{4}(u^2-1)^2\Big)
\]
be the Allen-Cahn energy associated to \eqref{acb1}.
To minimize $E_{\mathrm{AC}}$ the right balance between the contributions of the kinetic energy $\frac{1}{2}|u'|^2$ and the potential should be achieved. On the one hand the term $\frac{1}{2}|u'|^2$ penalizes high variations of $u$, while on the other hand the potential term $W$ forces the minimizer to be close to its global minima $\pm 1$. It is clear that the trivial solutions $\pm 1$ are the two global minimizers of $E_{\mathrm{AC}}$. Thus, it is more relevant to investigate instead, the existence of \emph{local minimizers} which are also called \emph{minimal} solutions. While solutions of \eqref{acb1} are critical points of $E_{\mathrm{AC}}$, a minimal solution $u$ of (\ref{acb1}) satisfies the stronger condition:
\[
E_{\mathrm{AC}}(u, \mathrm{supp}\, \phi)\leq E_{\mathrm{AC}}(u+\phi, \mathrm{supp}\, \phi),
\]
for all $\phi\in C^\infty_0(\R)$ (i.e. any perturbation with compact support of $u$ has greater or equal energy). It turns out that up to translations and change of $x$ by $-x$, the only minimal solution of  \eqref{acb1} is the heteroclinic orbit $\eta(x)= \tanh(x/\sqrt{2})$, connecting at $\pm\infty$ the two phases $\pm 1$.

A much more challenging problem is the description of all bounded solutions of the Allen-Cahn P.D.E.:
\begin{equation}\label{acb2}
\Delta u=u^3-u, \qquad \mbox{in}\ \R^n,
\end{equation}
which is associated to the functional $E_{\mathrm{AC}}(u,\Omega):=\int_\Omega\big( \frac{1}{2}|\nabla u|^2+\frac{1}{4}(u^2-1)^2\big)$ (where $\Omega \subset \R^n$ is bounded).
De Giorgi in 1978 \cite{dgiorgi1} suggested a striking analogy with minimal surface theory that led to significant developments in P.D.E. and 
the Calculus of Variations, by stating the following conjecture about bounded solutions on $\R^n$:
\begin{conjecture}[De Giorgi]
	Let $u\in C^2(\R^n)$ be a solution to \eqref{acb2} such that 
	\begin{enumerate}
		\item $|u|<1$, 
		\item $\frac{\partial u}{\partial x_n}>0$ for all $x \in \R^n$.
	\end{enumerate}
	Is it true that all the level sets of $u$ are hyperplanes, at least for $n\leq 8$?
\end{conjecture}
The relationship with the Bernstein problem for minimal graphs is the reason why $n\leq 8$ appears in the conjecture. 
We refer to the expository papers of Farina and Valdinoci \cite{fav1}, and Savin \cite{s1} for a detailed account. 
The conjecture was proved by Ghoussoub and Gui in \cite{goussguic9} for $n=2$, for $n=3$ by Ambrosio and Cabr\'e in \cite{ambcac9} and for $4\leq n\leq 8$ by Savin in \cite{MR2480601} under the extra requirement that
 \begin{equation}\label{ConvInf}
 \lim_{x_n\rightarrow\pm\infty}u(x_1,\ldots,x_n)=\pm 1.
 \end{equation}
If we drop the monotonicity requirement as well as \eqref{ConvInf} and simply ask about the structure of minimal solutions\footnote{Again, we say that the solution $u$ is minimal if $E_{\mathrm{AC}}(u, \mathrm{supp}\, \phi)\leq E_{\mathrm{AC}}(u+\phi, \mathrm{supp}\, \phi)$, for all $\phi\in C^\infty_0(\R^n)$.}
of \eqref{acb2}, then we know from \cite{MR2480601} that, for $n\leq 7$ any minimal solution $u$ of \eqref{acb2} is either trivial i.e. $u\equiv\pm1$ or one dimensional i.e. $u(x)= \eta(  (x-x_0)\cdot \nu)$, for some $x_0\in\R^n$, and some unit vector $\nu\in\R^n$.
Thus the heteroclinic orbit $\eta$ of O.D.E. \eqref{acb1} plays a crucial role for entire solutions of P.D.E. \eqref{acb2}.

In order to construct other connecting solutions of \eqref{acb2}, one shall impose some additional requirements. For instance, when $n=2$, \eqref{acb2} admits a unique \emph{saddle} solution $u$ (cf. \cite{fife}) satisfying the following properties:  
\begin{itemize}
	\item $u(x_1,x_2)$ has the same sign as the product $x_1x_2$, 
	\item $u$ is odd with respect to $x_1$ and $x_2$,
	\item $\lim_{x_1\to \infty}u(x_1,x_2)=\eta(x_2)$, and $u_{x_1}(x_1,x_2)>0$, $\forall x_2>0$,
	\item $\lim_{\lambda\to \infty}u(\lambda \cos \theta,\lambda \sin \theta)=1$, $\forall \theta \in (0,\pi/2)$.
\end{itemize}
This example also outlines the hierarchical structure of \eqref{acb2}, since by taking the limit of a solution along certain directions at infinity, lower dimensional solutions are obtained (cf. \cite[Chapter 8]{book}). For more examples of connecting maps under symmetry assumptions or in the vector case, we refer to \cite[Chapters 6, 7, 9]{book} and the references therein (in particular
\cite{abgc9} and \cite{schac9}).

The second Painlev\'e O.D.E.: 
\begin{equation}
\label{pain 1d}
y''-xy-2y^3=0, \qquad x\in \R, 
\end{equation}
is basically obtained by multiplying the linear term in the right hand side of \eqref{acb1} by $-x$. 
Alternatively, we can write \eqref{pain 1d} as
\begin{equation}
\label{pain 1db}
y''=H_y(x,y), \qquad x\in \R, 
\end{equation}
with $H(x,y):=\frac{1}{2} x y^2 +\frac{1}{2} y^4$. In contrast with $W$ (defined below \eqref{acb1}), the potential $H$ is non autonomous i.e. it depends both on $x$ and $y$.

Equation \eqref{pain 1d} is known to play an important role in  the theory  of integrable systems \cite{MR1149378}, random matrices \cite{2006math.ph...3038D, Flaschka1980,2005math.ph...8062C},  Bose-Einstein condensates \cite{MR2062641, MR2772375, MR3355003,sourdis0} and other problems \cite{alikoakos_1,helffer1998,KUDRYASHOV1997397}. Recently \cite{Clerc2017} it has been  shown that when the right hand side of (\ref{pain 1d}) is allowed to be a constant $\alpha\in \R$ then it  describes local profiles of  the so{-}called {\it shadow  kink} in the theory of light-matter interaction of nematic liquid crystals (cf. also \cite{troy1}. \cite{sourdis2}). In \cite{clerc2, Clerc2018, panayotis_4} further relation between other types of non topological defects ({\it shadow vortices, shadow domain walls}) and the {\it generalized} Painlev\'e equation 
\begin{equation}\label{pain 0}
\Delta y-x_1 y-2y^3=0 , \qquad \forall x=(x_1,x_2)\in \R^2, 
\end{equation} 
was established  showing that their local structure is determined by special  solutions of (\ref{pain 0}). 
One of the characteristics of  these  solutions is that they should be entire, another is that they should  be {\it minimal}. To explain what this means, let $\Omega\in \R^2$ be a bounded subset of $\R^2$ and 
\[
E_{\mathrm{P_{II}}}(u, \Omega)=\int_\Omega \left[ \frac{1}{2} |\nabla u|^2 +\frac{1}{2}  x_1 u^2 +\frac{1}{2} u^4\right],
\]
be the functional associated to the generalized second Painlev\'e equation.
By definition  a  solution of (\ref{pain 0}) is minimal if
\begin{equation}\label{minnn}
E_{\mathrm{P_{II}}}(y, \mathrm{supp}\, \phi)\leq E_{\mathrm{P_{II}}}(y+\phi, \mathrm{supp}\, \phi)
\end{equation}
for all $\phi\in C^\infty_0(\R^2)$.  This notion of minimality is standard for many problems in which the energy of a localized solution is actually infinite due to non compactness of the domain. The study of minimal solutions of \eqref{pain 1d} has been  recently  initiated in \cite{Clerc2017} where we have showed that the Hastings-McLeod solution, denoted in this paper by  $h$,  is, up to the sign change,  the only minimal solution which is bounded at $+\infty$. We recall (cf. \cite{MR555581}) that $h:\R\to\R$ is positive, strictly decreasing ($h' <0$) and such that 
\begin{align}\label{asy0}
h(x)&\sim \mathop{Ai}(x), \qquad x\to \infty, \nonumber \\
h(x)&\sim \sqrt{|x|/2}, \qquad x\to -\infty.
\end{align}
Clearly, the asympotic behaviour of $h$ is determined by the location of the global minima of the potential $H(x,y)$ associated to the equation \eqref{pain 1d}. Indeed for $x$ fixed, $H$ attains its global minimum equal to $0$ when  $y=0$ and $x\geq 0$, and equal to $-\frac{x^2}{8}$ when $y=\pm \sqrt{|x|/2}$ and $x<0$. Thus, the global minima of $H$ bifurcate from the origin, and the two minimal solutions $\pm h$ of \eqref{pain 1d} interpolate these two branches of minima.

Equation \eqref{pain 0} or equivalently $\Delta y=H_y(x_1,y)$, with $x=(x_1,x_2)\in \R^2$ and $H(x_1,y):=\frac{1}{2} x_1 y^2 +\frac{1}{2} y^4$ (cf. the expression of $E_{\mathrm{P_{II}}}$), involves a non autonomous potential which is bistable for every fixed $x_1<0$. Hence the Painlev\'e generalized equation \eqref{pain 0} describes as the Allen-Cahn equation a phase transition model. For the later the phase transition connects the two minima $\pm1$ while for the former the phase transition connects the two branches $\pm \sqrt{(-x_1)^+/2}$ of minima of $H$ parametrized by $x_1$. Note that in the Painlev\'e model the phase transition occurs only in the P.D.E. case, i.e. when the domain is $\R^{n+1}=\R\times \R^n$ with $n\geq 1$. 
The scope of this paper is to construct a solution $y$ of (\ref{pain 0}) connecting as $x_2\to \pm \infty$ and $x_1$ is fixed, the two branches of minima of $H$ (cf. Theorem \ref{corpain} below). It is the first to our knowledge example of solution of the generalized Painlev\'e equation both relevant from the applications point of view and mathematically interesting. 
The solution $y$ has similar properties as the heteroclinic orbit $\eta$: it is odd and monotonous with respect to $x_2$, and as $x_1\to -\infty$ its rescaled profile is actually  given by $\eta$. After the statement of Theorem \ref{corpain}, we will further discuss its similarities with the heteroclinic orbit $\eta$.

\begin{theorem}\label{corpain}
There exists a solution $y:\R^2\to\R$ to 
\begin{equation}\label{painhom}
\Delta y-x_1 y-2y^3=0, \qquad \forall x=(x_1,x_2)\in \R^2,
\end{equation}
such that
\begin{itemize}
\item[(i)] $y$ is positive in the upper-half plane and odd with respect to $x_2$ i.e. $y(x_1,x_2)=-y(x_1,-x_2)$.
\item[(ii)] $y$ and its derivatives are  bounded in the half-planes $[s_0,\infty)\times \R$, $\forall s_0\in\R$. 
\item[(iii)] $y$ is minimal with respect to perturbations $\phi\in C^\infty_0(\R^2)$ such that $\phi(x_1,x_2)=-\phi(x_1,-x_2)$.
\item[(iv)] $\frac{|y(x_1,x_2)|}{\mathop{Ai}(x_1)}=O(1)$, as $x_1\to\infty$ (uniformly in $x_2$).
\item[(v)] For every $x_2\in\R$ fixed, let $\tilde y(t_1,t_2):=\frac{\sqrt{2}}{(-\frac{3}{2} t_1)^{\frac{1}{3}}}\, y\big(-(-\frac{3}{2} t_1)^{\frac{2}{3}}, x_2+t_2(-\frac{3}{2} t_1)^{-\frac{1}{3}}\big)$. Then
\begin{equation}
\lim_{l\to -\infty} \tilde y(t_1+l,t_2)=
\begin{cases}
\tanh(t_2/\sqrt{2})  &\text{when } x_2=0, \\
1  &\text{when } x_2>0, \\
-1  &\text{when } x_2<0,
\end{cases}
\end{equation} for the $C^1_{\mathrm{ loc}}(\R^2)$ convergence.
\item[(vi)] $y_{x_1}(x_1,x_2)<0$, $\forall x_1\in\R$, $\forall x_2>0$.
\item[(vii)] $y_{x_2}(x_1,x_2)>0$, $\forall x_1, x_2\in\R$, and $\lim_{l\to\pm\infty} y(x_1,x_2+l)=\pm h(x_1)$ in $C^2_{\mathrm{loc}}(\R^2)$, where $h$ is the Hastings-McLeod solution of (\ref{pain 1d}).
\end{itemize}
\end{theorem}

The solution provided by Theorem \ref{corpain} has a form of a quadruple connection between the Airy function $\mathop{Ai}$, the two one dimensional Hastings-McLeod solutions $\pm h$, and the heteroclinic orbit $\eta$ of the one dimensional Allen-Cahn equation. Comparing (iv) with (\ref{asy0}) we see that  as $x_1\to \infty$ the function $y(x_1,x_2)$ behaves  similarly as the Hastings-McLeod solution $h(x_1)$. At the same time, as $x_2\to\pm\infty$ we have $y(x_1, x_2)\to \pm h(x_1)$,  $x_2\to \pm\infty$. 
Perhaps the most interesting aspect of the above solution $y$ is stated in property (v), since after rescaling we obtain as $x_1\to-\infty$, the convergence to the heteroclinic orbit $\eta(x)= \tanh(x/\sqrt{2})$ of the Allen-Cahn O.D.E. \eqref{acb1}.
In the proof of Theorem \ref{corpain} it is shown that a minimal solution of  
\eqref{painhom} rescaled as in (v), converges as $x_1\to-\infty$ to a minimal solution of 
\eqref{acb2}. This deep connection of the structure of the Painlev\'{e} equation with the Allen-Cahn P.D.E., suggests that several properties of the Allen-Cahn equation should be transfered to the Painlev\'{e} equation.
Although by construction the solution $y$ is only minimal for odd perturbations, we expect that $y$ is actually minimal for general perturbations, and plays a similar role that the heteroclinic orbit for the Allen-Cahn equation. What's more  the two global minimizers $\pm 1$ of the functional $E_{\mathrm{AC}}$ have their counterparts in the two minimal solutions $\pm h$ of the Painlev\'{e} equation. Indeed,  property (vii) establishes that $y$ connects monotonically along the vertical direction $x_2$, the two minimal solutions $\pm h(x_1)$, in the same way that $\eta$ connects monotonically the two global minimizers $\pm 1$. While $\eta$ is a one dimensional object, the solution $y(x_1,x_2)$ is two dimensional, since $x_1$ parametrizes the branches of minima of the potential $H$, and only $x_2$ is involved in the phase transition.

We believe that in higher dimension $y:\R^{n+1}\to \R$, ($n\geq 1$)  the structure of solutions of \eqref{painhom} exactly mirrors that of \eqref{acb2}, and going further, one may ask: is it true that  that in dimension $n\leq 7$, any minimal solution $Y:\R^{n+1}\to\R$ of (\ref{painhom}) is either $Y(x_1,x_2,\ldots, x_{n+1})=\pm h(x_1)$ or $Y(x_1,x_2,\ldots, x_{n+1})=y(x_1, (x_2,\ldots,x_{n+1})\cdot \n+b)$, for some constant $b\in \R$, and some unit vector $\n \in\SF^{n-1}$ ?

\noindent\textit{Acknowledgments.}
M. G. Clerc was partially supported by Fondecyt 1180903.
M. Kowalczyk was partially supported by Chilean research grants Fondecyt 1130126 and 1170164, Fondo Basal AFB170001 CMM-Chile.
P. Smyrnelis was partially supported by Fondo Basal AFB170001 CMM-Chile and Fondecyt postdoctoral grant 3160055.
\section{Odd minimizers of the Ginzburg-Landau type energy}

We consider the energy functional
\begin{equation}
\label{funct 0}
E(u)=\int_{\R^2}\frac{\epsilon}{2}|\nabla u|^2-\frac{1}{2\epsilon}\mu(x)u^2+\frac{1}{4\epsilon}u^4,
\end{equation}
where $u\in  H^1(\R^2)$ and $\epsilon>0$ is a small parameters. 
We suppose that $\mu \in C^\infty(\R^2)$ is radial i.e. $\mu(x)=\mu_{\mathrm{rad}}(|x|)$, with $\mu_{\mathrm{rad}} \in C^\infty(\R)$ an even function. 
In addition we assume that 
\begin{equation}\label{hyp2}
   \mu \in L^\infty(\R^2),\ \mu_{\mathrm{rad}}'<0 \text{ in }(0,\infty), \text{ and $ \mu_{\mathrm{rad}}(\rho)=0$ for a unique $\rho>0$},\medskip \\
 \end{equation}
In the physical context described in \cite{clerc2} the function $\mu$ is specific
\[
\mu(x)=e^{\,-|x|^2}-\chi, \qquad \mbox{with some}\ \chi\in (0,1), \qquad f(x)=-\frac{1}{2}\nabla \mu(x),
\]
but this particular form is irrelevant here. The Euler-Lagrange equation of $E$ is
\begin{equation}\label{ode}
\ve^2 \Delta u+\mu(x) u-u^3=0,\qquad x=(x_1,x_2)\in \R^2,
\end{equation} 
and we also write its weak formulation:
\begin{equation}\label{euler}
\int_{\R^2} -\epsilon^2 \nabla u\cdot \nabla \psi+\mu u \psi-u^3 \psi=0,\qquad  \forall \psi \in H^1(\R^2),
\end{equation}
where $\cdot$ denotes the inner product in $\R^2$.
Note that due to the radial symmetry of $\mu$  the energy \eqref{funct 0} and equation 
\eqref{ode} are invariant under orthogonal transformations in the domain, and sign change in the range. Our strategy to construct the solution of (\ref{painhom}) enjoying the properties described in Theorem \ref{corpain} is to find first an {\it odd} with respect to $x_2$ minimizer $u_\epsilon$ of $E$ and then  scaling and passing to the limit $\epsilon\searrow 0$  recover $y$ - this gives us existence. Second, in section 
\ref{proofpain} we show all the properties of $y$ stated in Theorem \ref{corpain}.

We explain, formally at the moment, the relation between (\ref{painhom}) and the energy $E$. Looking at the energy density of  $E$ it is evident that as $\epsilon\to 0$ the modulus of the global or odd  minimizer $u_{\epsilon}$ should approach a nonnegative 
root of the polynomial
\[
-\mu(x)z+z^3=0,
\] 
or in other words, $|u_{\epsilon}|\to \sqrt{\mu^+}$ as $\epsilon\to 0$ in some, perhaps weak, sense. This function, called the Thomas-Fermi limit of the minimizer is not in $H^1(\R^2)$ and therefore the transition near the set $\mu(x)=0$ has to be mediated somehow. To see this let us consider a point $\xi$ such that $\mu(\xi)=0$. By (\ref{hyp2}) $\xi=\rho e^{\,i \theta}$.  At $\xi$ introduce local  orthogonal frame $(e^{\,i\theta}, i e^{\,i\theta})$ and coordinates $s=(s_1, s_2)$ associated with it.  Let $u_\epsilon$ be any solution of (\ref{ode}) and
\[
z(s)=\epsilon^{-1/3} u(\xi+\epsilon^{2/3} s).
\]
Noting that $\mu(\xi+\epsilon^{2/3} s)=\epsilon^{2/3} s_1\mu_1+\dots$ with $\mu_1<0$ we get that $z$ satisfies
\[
\Delta_s z+s_1\mu_1 z- z^3=o(1), \quad\mbox{as}\ \epsilon\searrow 0.
\]
The equation on the left becomes the second Painlev\'e equation after passing to the limit and suitable scaling. Now, suppose that $u_\epsilon$ is the odd minimizer of $E$, i.e. $u_\epsilon(x_1, x_2)=-u_\epsilon(x_1,- x_2)$. Except for the points $\bar x=( \pm \rho,0)$ the limiting function $z$ could  be one of the Hastings-McLeod one dimensional solutions. However, at $(\pm\rho,0)$ we should have $z(s_1, s_2)=-z(s_1, -s_2)$, which means that $z$ genuinely depends on both variables. This is the idea behind the proof of the existence part in Theorem \ref{corpain}. Showing properties of the solution is a different story and depends on rather tricky application of the moving plane method.

Our first purpose in this paper is to study qualitative properties of the global minimizers of $E$ as $\epsilon\searrow 0$. In our previous work \cite{Clerc2017}  we studied  the following energy
\[
E(u,\R)=\int_{\R}\frac{\epsilon}{2}|u_x|^2-\frac{1}{2\epsilon}\mu(x)u^2+\frac{1}{4\epsilon}|u|^4-a f(x)u,  \quad u\colon \R\to \R,
\]
where $a\geq 0$ is a parameter and $f=-\frac{1}{2} \mu'$, and in \cite{Clerc2018} we studied its analog for maps $u\colon\R^2\to \R^2$.

By proceeding as in \cite{Clerc2018}, one can see that under the above assumptions there exists a global minimizer $v$ of $E$ in $H^1(\R^2)$, namely 
that $E(v)=\min_{H^1(\R^2)} E$. 
In addition, we show that $v$ is a classical solution of \eqref{ode}, 
and $v$ is radial. 
Similarly, in the class $H^1_{\mathrm{odd}}(\R^2):=\{u \in H^1(\R^2): u(x_1,x_2)=-u(x_1,-x_2)\}$ of odd functions with respect to $x_2$, there exists an odd minimizer $u$ which also solves \eqref{ode} and satisfies $u(x_1,x_2)=u(-x_1,x_2)$.
Although in the sequel we will focus on the odd minimizer for completeness we chose to present our next result in a slightly more general framework. 
\begin{theorem}\label{theorem 1}
For $\epsilon\ll 1$ let $u_{\epsilon}$ be a solution of \eqref{ode} converging to $0$ as $|x|\to\infty$ (which may be  the odd or global minimizer).
Let $\rho>0$ be the zero of $\mu_{\mathrm{rad}}$ and let $\mu_1:=\mu_{\mathrm{rad}}'(\rho)<0$. For every $\xi=\rho e^{i\theta}$, we consider the local coordinates $s=(s_1,s_2)$ in the basis $(e^{i\theta},i e^{i\theta})$, and the 
rescaled functions:
\begin{equation}
\label{u scaled}
w_{\epsilon}(s)= 2^{-1/2}(-\mu_1\ve)^{-1/3} u_{\epsilon}\left( \xi+\ve^{2/3} \frac{s}{(-\mu_1)^{1/3}}\right).
\end{equation}
As $\ve\to 0$, the function $w_{\epsilon}$ converges in $C^2_{\mathrm{loc}}(\R^2)$ up to subsequence, 
to a function $y$ bounded in the half-planes $[s_0,\infty)\times \R$, for every $s_0\in\R$, which is a 
solution of
\begin{equation}\label{pain}
\Delta y(s)-s_1 y(s)-2y^3(s)=0, \qquad \forall s=(s_1,s_2)\in \R^2.
\end{equation}
In particular, if we take $u_\epsilon$ to be the odd  minimizer of $E$ and $\xi=(\pm\rho,0)$, then the solution $y$ satisfies $y(s_1, s_2)=-y(s_1,- s_2)$, and  is minimal with respect to perturbations $\phi\in C^\infty_0(\R^2)$, $\phi(s_1, s_2)=-\phi(s_1,- s_2)$. On the other hand, if we take $u_\epsilon$ to be the global minimizer then $y(s_1,s_2)= h(s_1)$ or $y(s_1,s_2)=- h(s_1)$. 
\end{theorem}
We observe  that as a  corollary of \cite[Theorem 1.1.]{panayotis_4}  it can be proven that  $|v_\epsilon|\to \sqrt{\mu^+}$  in 
$C^0_{\mathrm{loc}}(D(0; \rho))$. Because of the analogy between the functional $E$  and the Gross-Pitaevskii functional in the theory of Bose-Einstein condensates we will call $\sqrt{\mu^+}$ the Thomas-Fermi limit  of $v_\epsilon$. 
Theorem \ref{theorem 1} gives account on how non smoothness of  the limit  of $v_{\epsilon}$  is mediated near the circumference $|x|=\rho$, where 
$\mu$ changes sign, through the solution of (\ref{pain}). We should mention  here that  detailed description of the minimizers for yet more general setting of the energy can be found in \cite{Clerc2018, panayotis_4}. 

Before proving the theorem we gather general results for minimizers and solutions that are valid for any values of the parameters $\epsilon> 0$.  For the rest of this paper $v$ or $v_\epsilon$ will be the global minimizer and $u$ or $u_\epsilon$ will be the odd minimizer or a critical point of $E$. 
We first prove the existence of global and odd minimizers.

\begin{lemma}\label{lem exist min}
For every $\epsilon> 0$  there exists $v \in H^1(\R^2)$ such that $E(v)=\min_{H^1(\R^2)} E$. 
As a consequence, $v$ is a $C^\infty$ classical solution of \eqref{ode}. Moreover, for $\epsilon \ll 1$
the global minimizer $v$ is unique up to change of $v$ by $-v$, and it can be written as $v(x)=v_{\mathrm{rad}}(|x|)$, 
with $v_{\mathrm{rad}} \in C^\infty(\R)$, positive, even, and such that $\lim_{\infty}  v_{\mathrm{rad}}=0$.
\end{lemma}
\begin{proof}
We proceed as in \cite[Lemma 2.1]{Clerc2018} to establish that the global minimizer exists and is a smooth solution of \eqref{ode} converging to $0$ as 
$|x|\to \infty$.
Next, we notice that $v \not\equiv 0$ for $\epsilon \ll 1$. 
Indeed, by choosing a test function $\psi\not\equiv 0$ supported in $D(0;\rho)\cap\{ x_2>0\}$, and such that $\psi^2<2\mu$, one can see that 
\[
E(\psi)=\frac{\epsilon}{2}\int_{\R^2}|\nabla \psi|^2+\frac{1}{4\epsilon}\int_{\R^2}\psi^2(\psi^2-2\mu)<0, \qquad  \epsilon\ll 1.
\] 
Let $x_0\in \R^2$ be such that $v(x_0)\neq 0$. Without loss of generality we may assume that $v(x_0)>0$. Next, consider $\tilde v=|v|$ which is another global minimizer and thus another solution.
Clearly, in a neighborhood of $x_0$ we have $v=|v|$, and as a consequence of the unique continuation principle (cf. \cite{sanada})
we deduce that $v\equiv \tilde v\geq 0$ on $\R^2 $. Furthermore, 
the maximum principle implies that $v>0$, since $v\not\equiv 0$.
To prove that $v$ is radial we consider the reflection with respect to the line $x_1=0$.
We can check that $E(v,\{x_1>0\})=E(v,\{x_1<0\})$, since otherwise by even reflection we can construct a map in $H^1$ with energy smaller  than $v$. Thus, the map 
$\tilde v(x) = v(|x_1|,x_2)$ is also a minimizer, and since $\tilde v= v$ on $\{x_1>0\}$, it follows by unique continuation that 
$\tilde v\equiv v$ on $\R^2$. Repeating the same argument for any line of reflection, we deduce that $v$ is radial.
To complete the proof, it remains to show the uniqueness of $v$ up to change of $v$ by $-v$. Let $\tilde v$ be another global minimizer such that $\tilde v>0$, and $\tilde v\not\equiv v$. 
Choosing $\psi=u$ in \eqref{euler}, we find for any solution $u\in H^1(\R^2)$ of \eqref{ode} the following alternative expression of the energy:
\begin{equation}\label{enealt}
E(u)=-\int_{\R^2}  \frac{u^4}{4\epsilon}.
\end{equation}
Formula \eqref{enealt} implies that $v $ and $\tilde v$ intersect for $|x|=r>0$. However, setting
\begin{equation*}
w(x)=\begin{cases}
      v(x) &\text{ for } |x|\leq r\\
       \tilde v(x) &\text{ for } |x|\geq r,
     \end{cases}
 \end{equation*}
we can see that $w$ is another global minimizer, and again by the unique continuation principle we have $w\equiv v\equiv\tilde v$. 
This completes the proof of Lemma \ref{lem exist min}.
\end{proof}

On the other hand, in the class $H^1_{\mathrm{odd}}(\R^2):=\{u \in H^1(\R^2): u(x_1,x_2)=-u(x_1,-x_2)\}$ of odd functions with respect to $x_2$, there exists an odd minimizer with the following properties:
\begin{lemma}\label{lem exist min odd}
For every $\epsilon> 0$ there exists $u \in H^1_{\mathrm{odd}}(\R^2)$ such that $E(u)=\min_{H^1_{\mathrm{odd}}(\R^2)} E$. As a consequence, $u$ is a $C^\infty$ classical solution of \eqref{ode}. Moreover 
\begin{itemize}
\item[(i)] $u(x)\to 0$ as $|x|\to \infty$,
\item[(ii)] $u(x_1,x_2)=u(-x_1,x_2)$,
\item[(iii)] up to transformation  $u\mapsto -u$ we have  $u(x_1,x_2)>0$, $\forall (x_1,x_2)\in \R\times (0,\infty)$, provided that $\epsilon\ll 1$.
\end{itemize}
\end{lemma}
\begin{proof}
The existence of $u \in H^1_{\mathrm{odd}}(\R^2)$ such that $E(u)=\min_{H^1_{\mathrm{odd}}(\R^2)} E$, follows as in \cite[Lemma 2.1]{Clerc2018}, and clearly $u$ is a critical point of $E$ in the subspace $H^1_{\mathrm{odd}}(\R^2)$.
In view of the radial symmetry of $\mu$ it is easy to see that the Euler-Lagrange equation \eqref{euler} 
holds also for every $\phi \in  H^1(\R^2)$, such that $\phi(x_1,x_2)=\phi(x_1,-x_2)$. As a consequence, 
$u$ is a $C^\infty$ classical solution of \eqref{ode}.

For the proof of (i) we refer to \cite[Lemma 2.1]{Clerc2018}.
To show that $u(x_1,x_2)=u(-x_1,x_2)$, we first note that $E(u,[0,\infty)\times\R )=E(u,(-\infty,0]\times\R)$. Indeed, if we assume without loss of generality that $E(u,[0,\infty)\times\R)<E(u, (-\infty,0]\times\R)$, the function
\begin{equation}
\tilde u(x_1,x_2)=
\begin{cases}
u(x_1,x_2) &\text{when  } x_1\geq 0,\\
u(-x_1,x_2) &\text{when  } x_1\leq 0,
\end{cases}
\end{equation}
has strictly less energy than $u$, which is a contradiction. Thus, $E(u,[0,\infty)\times\R )=E(u,(-\infty,0]\times\R)$, and as a consequence the 
function $\tilde u$ is also an odd minimizer and a solution. It follows by unique continuation \cite{sanada} that $\tilde u\equiv u$, that is, $u(x_1,x_2)=u(-x_1,x_2)$.

Now, it remains
to establish the uniqueness of the odd minimizer $u$, when $\epsilon\ll 1$.
Proceeding as in Lemma \ref{lem exist min}, we can see that $u \not\equiv 0$ for $\epsilon \ll1 $, and that either $u>0$ or $u<0$ on $\R\times (0,\infty)$.
Assume that $u_1$ and $u_2 $ are two minimizers of $E$ in $H^1_{\mathrm{odd}}(\R^2)$ such that 
$u_1>0$ and $u_2>0$ on $\R\times (0,\infty)$. Next, define the maps
\begin{equation}
 u_*(x_1,x_2)=
\begin{cases}
\min(u_1(x_1,x_2),u_2(x_1,x_2)) &\text{when  } x_2\geq 0,\\
\max(u_1(x_1,x_2),u_2(x_1,x_2)) &\text{when  } x_2\leq 0,
\end{cases}
\end{equation}
\begin{equation}
 u^*(x_1,x_2)=
\begin{cases}
\max(u_1(x_1,x_2),u_2(x_1,x_2)) &\text{when  } x_2\geq 0,\\
\min(u_1(x_1,x_2),u_2(x_1,x_2)) &\text{when  } x_2\leq 0,
\end{cases}
\end{equation}
and the set $A:=\{(x_1,x_2)\in\R\times(0,\infty): u_1(x_1,x_2)<u_2(x_1,x_2)\}$. We can see that $E(u_1,A)=E(u_2,A)$ since otherwise we have either $E(u_*)<E(u_2)$ or
$E(u^*)<E(u_1)$, which contradicts the minimality of $u_1$ and $u_2$. As a consequence, $E(u_*)=E(u_2)=E(u_1)=E(u^*)$, and it follows that $u_*$ and $u^*$ are also minimizers and solutions. Next, by unique continuation \cite{sanada}, we obtain that either $u_1\equiv u_*$ or $u_1\equiv u^*$, i.e. we have either 
$0\leq u_1\leq u_2$ or $u_1\geq u_2\geq 0$ on $ \R\times[0,\infty)$.
Finally, applying \eqref{enealt} to $E(u_1)=E(u_2)$, we conclude in view of the ordering of $u_1$ and $u_2$ that $u_1\equiv u_2$.
This completes the proof.
\end{proof}
To study the limit of solutions as $\epsilon\to 0$, we need uniform bounds.
Modifying slightly  the arguments in \cite[Section 2]{Clerc2018}, we obtain:
\begin{lemma}\label{s3}
For every  $\epsilon>0$  let $u_{\epsilon}$ be a solution of \eqref{ode} converging to $0$ as $|x|\to \infty$. Then, $u_{\epsilon}$ are uniformly bounded. 
\end{lemma}
\begin{proof}
We drop the index and write $u:=u_{\epsilon}$.
Since $\mu$ is bounded, the roots of the cubic equation $u^3-\mu(x)u=0$ belong to a bounded interval, for all values of $x$. 
If $u$ takes positive values, then it attains its maximum $0\leq \max_{\R^2}u=u(x_0)$, at a point $x_0\in\R^2$. 
In view of \eqref{ode}: $$0\geq\epsilon^2 \Delta u(x_0)=u^3(x_0)-\mu(x_0)u(x_0),$$ thus it follows that 
$u(x_0)$ is uniformly bounded above. In the same way, we prove the uniform lower bound for $u$. 
\end{proof}

\begin{lemma}\label{l2}
For $\epsilon\ll 1$  let $u_{\epsilon}$ be a solution of \eqref{ode} converging to $0$ as $|x|\to\infty$. 
Then, there exist a constant  $K>0$ such that 
\begin{equation}\label{boundd}
|u_{\ve}(x)|\leq K(\sqrt{\max(\mu (x),0)}+\ve^{1/3}), \quad \forall x\in \R^2.
\end{equation}
As a consequence, if for every $\xi=\rho e^{i\theta}$ we consider the local coordinates $s=(s_1,s_2)$ in the basis $(e^{i\theta},i e^{i\theta})$, then the rescaled functions $w_\epsilon(s)$ defined  in (\ref{u scaled})  are uniformly bounded on the half-planes $[s_0,\infty)\times\R$, $\forall s_0\in\R$. 
\end{lemma}

\begin{proof}
As above we write $u:=u_{\epsilon}$. Let us define the following constants
\begin{itemize}
\item $ M>0$ is the uniform bound of $|u_{\epsilon}|$ (cf. Lemma \ref{s3}),
\item $\lambda>0$ is such that $3 \mu_{\mathrm{rad}}(\rho-h)\leq 2\lambda h$, $\forall h \in [0,\rho]$,
\item $\kappa>0$ is such that  $\kappa^4\geq 6\lambda$.
\end{itemize}
Next, we construct the following comparison function
\begin{equation}\label{compchi}
\chi(x)=
\begin{cases}
\lambda\Big(\rho-|x|+\frac{\epsilon^{2/3}}{2}\Big)&\text{ for } |x|\leq\rho,\\
\frac{\lambda}{2\epsilon^{2/3}}(|x|-\rho-\epsilon^{2/3})^2&\text{ for }\rho\leq |x|\leq\rho+\epsilon^{2/3},\\
0&\text{ for } |x|\geq\rho+\epsilon^{2/3}.
\end{cases}
\end{equation}
One can check that $\chi\in C^1(\R^2\setminus\{0\})\cap H^1(\R^2)$ satisfies $\Delta \chi\leq \frac{2\lambda}{\epsilon^{2/3}}$ in $H^1(\R^2)$. Finally, we define the function
$\psi:=\frac{|u|^2}{2}-\chi-\kappa^2\epsilon^{2/3}$, and compute:
\begin{align}
\epsilon^2 \Delta \psi&=\epsilon^2 (|\nabla u|^2+ u\Delta u-\Delta \chi)\nonumber\\
&\geq-\mu |u|^2+|u|^4-\epsilon^2 \Delta \chi \nonumber\\
&\geq-\mu |u|^2+|u|^4-2\epsilon^{4/3}\lambda.
\end{align}
Now, one can see that when $x\in \omega:=\{x\in\R^2: \psi(x)> 0\}$, we have $\frac{|u|^4}{3}- \mu|u|^2\geq 0$, since
$$x \in\omega\cap \overline{D(0;\rho)}\Rightarrow \frac{|u|^4}{3}\geq \frac{2\lambda}{3}\Big(\rho-|x|+\frac{\epsilon^{2/3}}{2}\Big)|u|^2\geq \mu|u|^2 .$$
In the open set $\omega$ we also have: $\frac{|u|^4}{3}\geq \frac{\kappa^4}{3}\epsilon^{4/3}\geq 2\epsilon^{4/3}\lambda$, 
thus $\Delta \psi \geq 0$ in $\omega$ in the $H^1$ sense. 
To conclude, we apply Kato's inequality that gives: $\Delta \psi^+ \geq 0$ on $\R^2$ in the $H^1$ sense. Since $\psi^+$ is subharmonic with compact
support, we obtain by the maximum principle that $\psi^+\equiv 0$ or equivalently $\psi \leq 0$ in $\R^2$. The statement of the lemma follows by adjusting the constant $K$.
\end{proof}

After this preparation we are ready to prove the main result of this section.

\begin{proof}[Proof Theorem \ref{theorem 1}] 
For every $\xi=\rho e^{i\theta}$ we consider the local coordinates $s=(s_1,s_2)$ in the basis $(e^{i\theta},i e^{i\theta})$,
and we rescale the solutions by setting $\tilde u(s)=\frac{u_{\epsilon}(\xi+ s\epsilon^{2/3})}{\epsilon^{1/3}}$.
Clearly $\Delta \tilde u(s)=\epsilon \Delta u(\xi+s\epsilon^{2/3})$, thus,
\begin{equation}\label{oderes1}
\Delta \tilde u(s)+\frac{\mu(\xi+s\epsilon^{2/3})}{\epsilon^{2/3}} \tilde u(s)-\tilde u^3(s)=0, \qquad \forall s\in \R^2.\nonumber
\end{equation}
Writing $\mu(\xi+h)=\mu_1 h_1+h\cdot A(h)$, with $\mu_1:=\mu'_{\mathrm{rad}}(\rho)<0$, $A \in C^\infty(\R^2,\R^2)$, and $A(0)=0$, we obtain
\begin{equation}\label{oderes2}
\Delta \tilde u(s)+(\mu_1 s_1 + A(s \epsilon^{2/3})\cdot s) \tilde u(s)-\tilde u^3(s)=0,\qquad  \forall s\in \R^2.
\end{equation}
Next, we define the rescaled energy by
\begin{equation}
\label{functres2}
\tilde E(\tilde u)=\int_{\R^2}\left(\frac{1}{2}|\nabla\tilde u(s)|^2-\frac{\mu(\xi+s \epsilon^{2/3})}{2\epsilon^{2/3}}\tilde u^2(s)+\frac{1}{4}\tilde u^4(s)\right)\dd s.
\end{equation}
With this definition $\tilde E(\tilde u)=\frac{1}{\epsilon^{5/3}}E(u)$.
From Lemma \ref{l2} and \eqref{oderes2}, it follows that $\Delta \tilde u$, and also $\nabla\tilde u$, are uniformly bounded on compact sets. 
Moreover, by differentiating \eqref{oderes2} we also obtain the boundedness of the second derivatives of $\tilde u$.
Thanks to these uniform bounds, we can apply the theorem of Ascoli via a diagonal argument to obtain the convergence of $\tilde u$ in 
$C^2_{\mathrm{loc}}(\R^2)$ (up to a subsequence) 
to a solution $\tilde z $ of 
\begin{equation}\label{oderes4}
\Delta \tilde z(s)+\mu_1 s_1 \tilde z(s)-\tilde z^3(s)=0, \ \forall s\in \R^2, 
\end{equation}
which is associated to the functional
\begin{equation}
\label{functres4}
\tilde E_0(\phi,J)=\int_{J}\left(\frac{1}{2}|\nabla\phi(s)|^2-\frac{\mu_1}{2} s_1 \phi^2(s)+\frac{1}{4}\phi^4(s)\right)\dd s.
\end{equation}
Given $\tilde  \psi(s)$ a test function supported in the compact set $K$, let $\psi(x):=\epsilon^{1/3}\tilde \psi (\frac{x-\xi}{\epsilon^{2/3}})\Leftrightarrow \tilde \psi(s)=\frac{\psi(\xi+ s\epsilon^{2/3})}{\epsilon^{1/3}}$.
In the case where we take $u$ to be the global minimizer $v$, since
$ E( v_\epsilon+ \psi,\supp\psi)\geq  E( v_\epsilon,\supp \psi)$, 
we have 
$\tilde E(\tilde v_\epsilon+\tilde \psi,K)\geq \tilde E(\tilde v_\epsilon,K)$, and at the limit
$\tilde E_{0}( \tilde z+\tilde\psi,K)\geq  \tilde E_{0}(\tilde z,K)$.
Thus, $\tilde z$ is a minimal solution of \eqref{oderes4}. In addition, the radial symmetry of $v$, implies that
$\tilde z$ depends only on the variable $s_1$. Indeed, noticing that $\lim_{\epsilon\to 0}\frac{|\xi +\epsilon^{\frac{2}{3}}(s_1,s_2)|-\rho}{\epsilon^{\frac{2}{3}}}=s_1$,
it follows that $\tilde v_\epsilon (s_1,s_2)=\tilde v_\epsilon (s_1+o(1),0)$, and $\tilde z(s_1,s_2)=\tilde z(s_1,0)$.
Similarly, in the case where we take $u$ to be the odd minimizer and $\xi=(\pm\rho,0)$, we can see that $\tilde z$ is a minimal solution of \eqref{oderes4} for perturbations such that $\tilde \psi (s_1,s_2)=-\tilde \psi(s_1,-s_2)$.
Finally, setting $y(s):=\frac{1}{\sqrt{2}(-\mu_1)^{1/3}}\tilde z\big(\frac{s}{(-\mu_1)^{1/3}}\big)$, \eqref{oderes4} reduces to \eqref{pain}, 
that is, $y$ solves \eqref{pain}. In the case where we take $u$ to be the global minimizer $v$, we can see that either $y(s_1,s_2)=h(s_1)$ or $y(s_1,s_2)=-h(s_1)$, since $\pm h$ are the only minimal solutions of \eqref{pain 1d} (cf. \cite[Theorem 1.3]{Clerc2017}). On the other hand, in the case where we take $u$ to be the odd minimizer and $\xi=(\pm\rho,0)$, it is clear that  
$y$ is odd with respect to $s_2$, and minimal for perturbations such that $\tilde \psi (s_1,s_2)=-\tilde \psi(s_1,-s_2)$.
\end{proof}

\section{Proof of Theorem \ref{corpain}}\label{proofpain}
We will proceed in few steps.  The proof of (i), (ii) and (iii) follows from Theorem \ref{theorem 1}, Lemma \ref{l2}, and the fact that a minimal solution of \ref{painhom} cannot be identically zero. To establish (v) we proceed as in Theorem \ref{theorem 1}. After rescaling appropriately $y$ as $x_1\to-\infty$, we compute uniform bounds of the rescaled functions. Then, by the theorem of Ascoli, we obtain at the limit a minimal solution of the Allen-Cahn equation \eqref{acb2}. The proof of (vi) and (vii) is based on the moving plane method applied in a sector contained in the upper half-plane. The main difficulty is due to the unboundedness of the domain and to the availability of boundary conditions only on the $x_1$ axis where $y(x_1,0)=0$. We also utilize the asymptotic behaviour of $y$, as $x_1\to \pm\infty$, provided respectively by (v) and Lemma \ref{expcvv}. Our main tool is a version of the maximum principle in unbounded domains (cf. Lemma \ref{maxpp}), that allows us to compute bounds for $y_{x_1}$ and $y_{x_2}$ when $x_1$ is large enough and $x_2>0$ (cf. Lemmas \ref{lll1} and \ref{lll2}). Next, these bounds are extended to the whole half-plane $x_2>0$ by applying the sliding method (cf. Lemma \ref{movingplane}).

\begin{proof}[Proof of (i), (ii) and (iii)]
By applying Theorem \ref{theorem 1} in a neighborhood of the point $\xi=(\rho, 0)$ to the odd minimizer $u$, such that $u>0$ on $\R\times (0,\infty)$,  it is clear that we obtain a solution $y$ of \eqref{pain} which is odd with respect to the second variable $s_2$, and such that $y\geq 0$, on $\R\times (0,\infty)$. For the sake of convenience in what follows we substitute the variables $(s_1,s_2)$ by $(x_1,x_2)$. The properties (ii) and (iii) are also straightforward by Theorem \ref{theorem 1} and Lemma \ref{l2}. Thus, it remains to show that
$y(x_1,x_2)> 0$, $\forall x \in\R\times (0,\infty)$. Assume by contradiction that $y(x_1,x_2)= 0$, for some $x \in\R\times (0,\infty)$, then it follows from the maximum principle that $y\equiv 0$. To conclude we are going to show that a solution $y$ of \eqref{painhom}
which is minimal for odd perturbations, cannot be identically zero. Indeed, the minimality of $y$ implies that the second variation of the energy $E_{\mathrm{P_{II}}}$ 
is nonnegative:
\begin{equation}\label{secondvar}
\int_{\R^2} (|\nabla \phi(x)|^2+(6 y^2(x)+x_1)\phi^2(x))\dd x \geq 0,\quad  \forall \phi\in C^1_0(\R^2), \text{ such that } \phi(x_1,x_2)=-\phi(x_1,-x_2).
\end{equation}
Clearly \eqref{secondvar} does not hold when $y\equiv 0$, if we take $\phi(x)=\phi_0(x_1+l,x_2)$, with $l\to\infty$, and $\phi_0\in C^1_0(\R^2)$ fixed, such that $\phi_0(x_1,x_2)=-\phi_0(x_1,-x_2)$, and $\phi_0 \not\equiv 0$. 
\end{proof}

Next we recall a useful version of the maximum principle in unbounded domains \cite[Lemma 2.1]{beres}.
\begin{lemma}\label{maxpp}
Let $D$ be a domain (open connected set) in $\R^n$, possibly unbounded. Assume that $\overline D$ is disjoint from the closure of an infinite open connected cone $\Sigma$. Suppose there is a function $z$ in $C(\overline D)$ that is bounded above and satisfies for some continuous function $c(x)$
$$\Delta z-c(x)z\geq 0 \text{ in $D$ with } c(x)\geq 0$$
$$z\leq 0 \text{ on }\partial D.$$
Then $z\leq 0 $ in $D$. 
\end{lemma}

As a first application of Lemma \ref{maxpp} we prove the exponential convergence of $y$ to $0$, as $x_1\to \infty$.
\begin{lemma}\label{expcvv}
$|y(x_1,x_2)|=O(e^{-\frac{2}{3}x_1^{3/2}})$, as $x_1\to\infty$ (uniformly in $x_2$).
\end{lemma}
\begin{proof}
We define $\psi(x_1,x_2):=M e^{-\frac{2}{3} x_1^{3/2}}$, in the domain $D:=\{(x_1,x_2): \, x_1>1, x_2>0\}$, where $M\geq e^{\frac{2}{3}}\sup_{x_2\geq 0} y(1,x_2)$ is a constant. It is easy to see that $\Delta \psi \leq x_1\psi$ in $D$, and $\Delta (y-\psi) \geq x_1(y-\psi)$ in $D$. Since $y-\psi\leq 0$ on $\partial D$, it follows from Lemma \ref{maxpp} that $y\leq \psi$ in $D$.
\end{proof}
\begin{proof}[Proof of (v)]
We set $(t_1,t_2):=\big(-\frac{2}{3} (-x_1)^{\frac{3}{2}}, (-x_1)^{\frac{1}{2}}r\big)$, where $x_1\leq -1$ and $r\in\R$. Equivalently we have $(x_1,r)=\big(-(-\frac{3}{2} t_1)^{\frac{2}{3}}, t_2(-\frac{3}{2} t_1)^{-\frac{1}{3}}\big)$. Next we define $\tilde y(t_1,t_2):=\frac{\sqrt{2}}{(-\frac{3}{2} t_1)^{\frac{1}{3}}}\, y(x_1, r+x_2)$, for every $x_2\in\R$ \emph{fixed}, or equivalently
\begin{equation}\label{eqder}
 y(x_1, r+x_2)=\frac{(-x_1)^{\frac{1}{2}}}{\sqrt{2}}\tilde y(t_1,t_2).
\end{equation}
We are going to show that $\tilde y(t_1,t_2)$ is uniformly bounded up to the second derivatives, when $t_2$ belongs to a compact interval and $t_1\to-\infty$.
By differentiating \eqref{eqder} with respect to $s_1$ and $r$ we obtain
\begin{subequations}\label{eqderder}
\begin{equation}\label{eqder1}
\sqrt{2}y_{x_2}(x_1, r+x_2)=(-x_1)\tilde y_{t_2}(t_1,t_2), 
\end{equation}
\begin{equation}\label{eqder2}
\sqrt{2}y_{x_2x_2}(x_1, r+x_2)=(-x_1)^{\frac{3}{2}}\tilde y_{t_2t_2}(t_1,t_2), 
\end{equation}
\begin{equation}\label{eqder3}
 \sqrt{2}y_{x_1}(x_1, r+x_2)=-\frac{1}{2}(-x_1)^{-\frac{1}{2}}\tilde y(t_1,t_2)+(-x_1) \tilde y_{t_1}(t_1,t_2)-\frac{r}{2}\tilde y_{t_2}(t_1,t_2), 
\end{equation}
\begin{equation}\label{eqder4}
 \sqrt{2}y_{x_1x_2}=-\tilde y_{t_2}+(-x_1)^{\frac{3}{2}}\tilde y_{t_1t_2}-\frac{r}{2}(-x_1)^{\frac{1}{2}}\tilde y_{t_2,t_2}, 
\end{equation}
\begin{equation}\label{eqder5}
 \sqrt{2}y_{x_1x_1}=-\frac{1}{4}(-x_1)^{-\frac{3}{2}}\tilde y-\frac{3}{2}\tilde y_{t_1} +\frac{r}{4} (-x_1)^{-1}\tilde y_{t_2}
+(-x_1)^{\frac{3}{2}}\tilde y_{t_1t_1}-r(-x_1)^{\frac{1}{2}}\tilde y_{t_1t_2}+\frac{r^2}{4}(-x_1)^{-\frac{1}{2}}\tilde y_{t_2,t_2}. 
\end{equation}
\end{subequations}
Since by construction (cf. \eqref{boundd} in Lemma \ref{l2}) $y$ satisfies $|y(x_1,x_2)|=O(|-x_1|^{\frac{1}{2}})$ as $x_1\to-\infty$ (i.e. $\tilde y$ is bounded), we obtain by \eqref{painhom} and standard elliptic estimates \cite[\S 3.4 p. 37 ]{1987130} that 
\begin{equation}\label{eqder6}
\text{$|\nabla y(x_1,x_2)|=O(|-x_1|^{\frac{3}{2}})$ and $|D^2 y(x_1,x_2)|=O(|-x_1|^{\frac{5}{2}})$, as $x_1\to-\infty$.}
\end{equation}
From \eqref{eqder6} and \eqref{eqderder} it follows that
\begin{equation}\label{eqder7}
\text{$|\nabla \tilde y(t_1,t_2)|=O(|-x_1|^{\frac{1}{2}})$ and $|D^2 \tilde y(t_1,t_2)|=O(|-x_1|)$, as $x_1\to-\infty$,}
\end{equation}
provided that $(t_1,t_2)\in \Sigma_{t_0,r_0}:=\{ (t_1,t_2): t_1\leq t_0, \, |t_2|\leq r_0 (-\frac{3}{2} t_1)^{\frac{1}{3}}\}$, where $t_0<0$ and $r_0>0$ are arbitrary constants. In particular, we have $\sqrt{2}\Delta y(x_1,x_2)=(-x_1)^{\frac{3}{2}}\Delta \tilde y(t_1,t_2)+O(|-x_1|^{\frac{3}{2}})$, for $(t_1,t_2)\in \Sigma_{t_0,r_0}$. On the other hand it is clear by \eqref{painhom} that 
$\sqrt{2}\Delta y(x_1,x_2)=(-x_1)^{\frac{3}{2}}(\tilde y^3(t_1,t_2)-\tilde y(t_1,t_2))$, thus
\begin{equation}\label{eqder8}
\text{$|\Delta \tilde y(t_1,t_2)|$ and $|\nabla\tilde y(t_1,t_2)|$ are bounded, $\forall (t_1,t_2)\in \Sigma_{t_0,r_0}$.}
\end{equation}
Similarly, by differentiating once more equations \eqref{eqderder} with respect to $x_1$ and $r$, one can show that
\begin{equation}\label{eqder9}
\text{$|D^2\tilde y(t_1,t_2)|$ is bounded, $\forall (t_1,t_2)\in \Sigma_{t_0,r_0}$.}
\end{equation}
Next, we apply the theorem of Ascoli to the sequence $\tilde y(t_1+ l,t_2)$ as $ l\to-\infty$. Up to a subsequence $l_n\to -\infty$, we obtain via a diagonal argument, the convergence in $C^1_{\mathrm{ loc}}(\R^2)$ of $\tilde y_n(t_1,t_2):=\tilde y(t_1+l_n,t_2)$ to a bounded function $\tilde z(t_1,t_2)$ that we are going to determine. 
Our claim is that the limit $\tilde z$ is a minimal solution of the Allen-Cahn equation \eqref{acb2}, which is independent of the subsequence $l_n$. The proof of this property is based on the following energy considerations. Let $(e_1,e_2)$ be the canonical basis of $\R^2$.
The energy functional 
\begin{equation}\label{enph}
E_{\mathrm{P_{II}}}(y, A)=\int_{A-x_2e_2} \left[ \frac{1}{2} |\nabla y(x_1,r+x_2)|^2 +\frac{1}{2}  x_1 y^2(x_1,r+x_2) +\frac{1}{2} y^4(x_1,r+x_2)\right]\dd x_1\dd r,
\end{equation}
associated to \eqref{painhom}, becomes after changing variables as in (\ref{eqder})
\begin{equation}\label{enph2}
E_{\mathrm{P_{II}}}(y, A)=\tilde E_{\mathrm{P_{II}}}(\tilde y, \tilde A)=\tilde F(\tilde y, \tilde A)+ \tilde R(\tilde y, \tilde A),
\end{equation}
where
\begin{equation}\label{tildaa}
\tilde A:=\{(t_1(x_1),t_2(x_1,r)): (x_1,r)\in A-x_2e_2 \},
\end{equation}
\begin{equation}\label{festo}
\tilde F(\tilde y, \tilde A):=\int_{\tilde A } \frac{1}{2} \Big(-\frac{3}{2} t_1\Big)^{\frac{2}{3}}\left[\frac{1}{2}|\nabla \tilde y(t_1,t_2)|^2 -\frac{\tilde y^2(t_1,t_2)}{2}  +\frac{\tilde y^4(t_1,t_2)}{4}\right] \dd t_1\dd t_2,
\end{equation}
and 
\begin{equation}\label{resto}
\tilde R(\tilde y, \tilde A):=\int_{\tilde A}\left[ \frac{(\tilde y +t_2 \tilde y_{t_2})^2}{16 (-\frac{3}{2} t_1)^{\frac{4}{3}}} -\frac{(\tilde y+t_2 \tilde y_{t_2})\tilde y_{t_1}}{4 (-\frac{3}{2} t_1)^{\frac{1}{3}}} \right]\dd t_1 \dd t_2.
\end{equation}
Let $\tilde \phi(t_1,t_2)\in C^\infty_0(\R^2)$ be a test function such that 
$\tilde B:=\supp \tilde \phi \subset\{(t_1,t_2): c-d\leq t_1\leq c\} $, for some constants $c \in\R$ and $d>0$.
Given $l\in\R$, we consider the translated functions $\tilde \phi^{-l}(t_1,t_2):=\tilde \phi(t_1-l,t_2)$, and 
$\tilde y^l(t_1,t_2)=\tilde y(t_1+l,t_2)$.
Note that $\tilde B^{l}:=\supp \tilde \phi^{-l}=\tilde B+l e_1$, and 
$\supp\tilde  \phi^{-l} \subset\{(t_1,t_2):  t_1<-1\} $ when $l<1-c$. Thus, for $l<1-c$, we can define $\phi^{-l}\in C^\infty_0(\R^2)$ by $\phi^{-l}(x_1, r+x_2)=\frac{(-x_1)^{\frac{1}{2}}}{\sqrt{2}}\tilde \phi^{-l}(t_1,t_2)$ as in\eqref{eqder}. 
Let $B^{l}:=\{(x_1(t_1),r(t_1,t_2)+x_2): \, (t_1,t_2)\in \tilde B^{l}\}$.

We first examine the case where $x_2=0$, and assume that $\tilde \phi(t_1,t_2)=-\tilde \phi (t_1,-t_2)$. In view of the minimality of $y$ and \eqref{enph2}, we have
\begin{equation}\label{enph3}
\tilde E_{\mathrm{P_{II}}}(\tilde y+\tilde\phi^{-l}, \tilde B^{l})=E_{\mathrm{P_{II}}}(y+\phi^{-l}, B^{l})\geq E_{\mathrm{P_{II}}}(y, B^{l})=\tilde E_{\mathrm{P_{II}}}(\tilde y, \tilde B^{l}).
\end{equation}
On the one hand, it is clear that the boundedness of $\tilde y$ and \eqref{eqder8} imply that $\lim_{l\to-\infty}\tilde R(\tilde y+\tilde\phi^{-l}, \tilde B^{l})=0$ and $\lim_{l\to-\infty}\tilde R(\tilde y, \tilde B^{l})=0$.
Next, setting $t_0:=c+l$, we have 
\begin{equation}\label{t000}
\Big(-\frac{3}{2} t_1\Big)^{\frac{2}{3}}\leq \Big(-\frac{3}{2} t_0\Big)^{\frac{2}{3}}+d \Big(-\frac{3}{2} t_0\Big)^{-\frac{1}{3}}, \forall t_1 \in [t_0-d,t_0].
\end{equation}
Thus, we obtain
\begin{equation}\label{diven1}
\tilde F(\tilde y, \tilde B^{l})=\frac{1}{2} \Big(-\frac{3}{2} t_0\Big)^{\frac{2}{3}}\tilde G(\tilde y,\tilde B^{l})+O(|t_0|^{-\frac{1}{3}})=\frac{1}{2} \Big(-\frac{3}{2} t_0\Big)^{\frac{2}{3}}\tilde G(\tilde y^l,\tilde B)+O(|t_0|^{-\frac{1}{3}}),
\end{equation}
and
\begin{equation}\label{diven2}
\tilde F(\tilde y+\tilde\phi^{-l}, \tilde B^{l})=\frac{1}{2} \Big(-\frac{3}{2} t_0\Big)^{\frac{2}{3}}\tilde G(\tilde y+\tilde\phi^{-l},\tilde B^{l})+O(|t_0|^{-\frac{1}{3}})=\frac{1}{2} \Big(-\frac{3}{2} t_0\Big)^{\frac{2}{3}}\tilde G(\tilde y^l+\tilde\phi,\tilde B)+O(|t_0|^{-\frac{1}{3}}),
\end{equation}
where we have set $\tilde G(\tilde z, \tilde B):=\int_{\tilde B}(\frac{1}{2}|\nabla \tilde z|^2 -\frac{\tilde z^2}{2}  +\frac{\tilde z^4}{4}) \dd t$. Finally, since $\tilde y^{l_n}(t_1,t_2)\to \tilde z(t_1,t_2)$ in $C^1_{\mathrm{ loc}}(\R^2)$, as $n \to \infty$, we conclude that
\begin{equation}\label{enph3concl}
\tilde G(\tilde z+\tilde\phi,\tilde B)=\lim_{n\to\infty}\frac{2}{(-\frac{3}{2} (c+l_n))^{\frac{2}{3}}}\tilde E_{\mathrm{P_{II}}}(\tilde y^{l_n}+\tilde\phi, \tilde B)\geq 
\lim_{n\to\infty}\frac{2}{(-\frac{3}{2} (c+l_n))^{\frac{2}{3}}}\tilde E_{\mathrm{P_{II}}}(\tilde y^{l_n}+\tilde\phi, \tilde B)=\tilde G(\tilde z,\tilde B),
\end{equation}
or equivalently $E_{\mathrm{AC}}(\tilde z+\tilde\phi,\tilde B)\geq E_{\mathrm{AC}}(\tilde z,\tilde B)$. This means that $\tilde z$ is a minimal solution of the Allen-Cahn equation \eqref{acb2} for odd perturbations $\tilde \phi$. In particular $\tilde z \not\equiv 0$, and as a consequence of the maximum principle, $\tilde z(t_1,0)=0$, $\forall t_1\in \R$, and $\tilde z(t_1,t_2)\geq 0$, $\forall (t_1,t_2)\in \R\times (0,\infty)$, imply that $\tilde z(t_1,t_2)> 0$, $\forall (t_1,t_2)\in \R\times (0,\infty)$. Thus, from \cite[Theorem 1.5]{beres2}, it follows that $\tilde z$ is a function of only $t_2$, which is the heteroclinic connection $\tilde z(t_1,t_2)=\eta(t_2)=\tanh(t_2/\sqrt{2})$. Furthermore, since the limit $\tilde z$ is uniquely determined, the convergence $\tilde y^l(t_1,t_2)\to\tilde z(t_1,t_2)$ holds as $l\to-\infty$.

It remains to examine the case where $x_2\neq 0$. Without loss of generality we assume that $x_2>0$. Now \eqref{enph3} holds for arbitrary test functions $\tilde \phi(t_1,t_2)\in C^\infty_0(\R^2)$, since $B^{l}\subset \{(x_1,x_2): x_2>0\}$ as $l\to-\infty$. Repeating the previous arguments we find that $\tilde z$ is a nonnegative minimal solution of \eqref{acb2}. Applying \cite[Corollary 5.2]{book}, we deduce that $\tilde z\equiv 1$. This completes the proof of (v).
\end{proof}

\begin{proof}[Proof of (vi) and (vii)]
The proofs of (vi) and (vii) which are based on the moving plane method, follow from the next lemmas.
\begin{lemma}\label{lll1}
We have $y_{x_1}(x_1,x_2)<0$, $\forall x_1\geq 0$, $\forall x_2>0$. In addition, for every $d>0$, there holds $\sup_{x_2\geq d}  y_{x_1}(1,x_2)<0$, and  $\inf_{x_2\geq d}  y(1,x_2)>0$.
\end{lemma}

\begin{proof}
Given $\lambda\geq 0$, we define the function $\psi_\lambda(x_1,x_2):=y(x_1,x_2)-y(-x_1+2\lambda,x_2)$ for $(x_1,x_2)\in D_\lambda:=\{ (x_1,x_2): x_1 > \lambda , x_2 >0\}$.  One can check that 
$\psi_\lambda=0$ on $\partial D_\lambda$, and $$\Delta \psi_\lambda-c(x_1,x_2)\psi_\lambda= 2(x_1-\lambda)y(-x_1+2\lambda,x_2)\geq 0,$$ with $c(x_1,x_2)=x_1+2(y^2(x_1,x_2)+y(x_1,x_2)y(-x_1+2\lambda,x_2)+y^2(-x_1+2\lambda,x_2))\geq 0$.
Furthermore, $\psi_\lambda$ is bounded above by Theorem \ref{corpain} (ii), and not identically zero by Theorem \ref{corpain} (v).
As a consequence of Lemma \ref{maxpp}, it follows that
$\psi_{\lambda}(x_1,x_2)< 0$, $\forall x_1> \lambda$, $\forall x_2> 0$, and thus by Hopf's Lemma $\frac{\partial \psi_\lambda}{\partial x_1}(\lambda,x_2)=2y_{x_1}(\lambda,x_2)<0$, $\forall x_2>0$. To establish that $\sup_{x_2\geq d}  y_{x_1}(1,x_2)<0$, we proceed by contradiction and assume the existence of a sequence $\{l_n\}$ such that $\lim_{n\to\infty} l_n=\infty$ and 
$\lim_{n\to\infty} y_{x_1}(1,l_n)=0$. Next, we set $\tilde y_n(x_1,x_2)=y(x_1,x_2+l_n)$. In view of the bounds provided in Theorem \ref{corpain} (ii), we obtain by the theorem of Ascoli that (up to subsequence) $\tilde y_n$ converges in $C^1_{\mathrm{loc}}$ to a nonnegative minimal solution $\tilde y$ of \eqref{painhom}. Since $\tilde y_{x_1}(1,0)=\lim_{n\to\infty}y_{x_1}(1,l_n)=0$, and $\tilde y_{x_1}(x_1,x_2)\leq 0$, $\forall x_1\geq 0$, $\forall x_2 \in \R$, the maximum principle applied to 
\begin{equation}\label{eqqq11}
\Delta \tilde y_{x_1}=\tilde y +(x_1+6 \tilde y^2)\tilde y_{x_1}\geq(x_1+6 \tilde y^2)\tilde y_{x_1},
\end{equation}
 implies that $\tilde y_{x_1}(x_1,x_2)= 0$, $\forall x_1\geq 0$, $\forall x_2 \in \R$. Then, since $\lim_{x_1\to\infty}\tilde y(x_1,x_2)=0$, $\forall x_2\in\R$, it follows that $\tilde y\equiv 0$ in the half-plane $x_1\geq 0$. Finally, we deduce by unique continuation that $\tilde y\equiv 0$ in $\R^2$, which is a contradiction since $\tilde y$ is minimal. Thus we have established that 
$\sup_{x_2\geq d}  y_{x_1}(1,x_2)<0$. The proof that $\inf_{x_2\geq d}  y(1,x_2)>0$ is similar.
\end{proof}
\begin{lemma}\label{lll2}
For every vector $\n=e^{i (\theta+\frac{\pi}{2})}\in \C\sim\R^2$, with $\theta \in (0,\frac{\pi}{2})$, there exists $s_\n>0$ such that $\nabla y(x_1,x_2)\cdot \n>0$, $\forall x_1 >s_\n$, $\forall x_2 >0$.
\end{lemma}

\begin{proof}
Our first claim is that there is a constant $k_1>0$, such that $k_1 y_{x_1} (x_1,x_2)\leq -\sqrt{x_1} y(x_1,x_2)$, $\forall x_1\geq 1$, $\forall x_2\geq 0$. Indeed, let $\psi(x_1,x_2)=k_1 y_{x_1}(x_1,x_2)+\sqrt{x_1} y(x_1,x_2)$ for $(x_1,x_2)\in D:=\{x_1>1, x_2>0\}$, where the constant $k_1$ will be adjusted later. It is clear that $\psi(x_1,0)=0$, $\forall x_1\geq 1$. We also note that
$y_{x_1x_2}(1,0)<0$, since  the function $y_{x_1}$ vanishes at $(1,0)$, is negative in $\{ x_1> 0,  x_2> 0\}$, and satisfies \eqref{eqqq11}. This and  $\sup_{x_2\geq d}  y_{x_1}(1,x_2)<0$, $\forall d>0$, imply that when $k_1$ is large enough, we have $\psi(1,x_2)\leq 0$, $\forall x_2\geq 0$. Next, we compute
\begin{align*}
\Delta \psi&=\Big(x_1+6y^2+\frac{1}{k_1\sqrt{x_1}}\Big)k_1y_{x_1}+\Big(x_1+2y^2+\frac{k_1}{\sqrt{x_1}}-\frac{1}{4x_1^2}\Big) \sqrt{x_1}y\\
&=\Big(x_1+2y^2+\frac{k_1}{\sqrt{x_1}}-\frac{1}{4x_1^2}\Big) \psi+\Big(4y^2+\frac{1}{k_1\sqrt{x_1}}-\frac{k_1}{\sqrt{x_1}}+\frac{1}{4x_1^2}\Big) k_1 y_{x_1}.
\end{align*}
By choosing $k_1$ large enough we can ensure that $\big(x_1+2y^2+\frac{k_1}{\sqrt{x_1}}-\frac{1}{4x_1^2}\big) \geq 0$ and $\big(4y^2+\frac{1}{k_1\sqrt{x_1}}-\frac{k_1}{\sqrt{x_1}}+\frac{1}{4x_1^2}\big)\leq 0$, when $x_1\geq 1$ and $x_2\geq 0$.
Thus, by applying Lemma \ref{maxpp}, our claim follows.

Similarly, we are going to establish that there is a constant $k_2>0$, such that $y_{x_2} (x_1,x_2)\geq -k_2 y(x_1,x_2)$, $\forall x_1\geq 1$, $\forall x_2\geq 0$. To do this we let $\psi(x_1,x_2)=- y_{x_2}(x_1,x_2)-k_2 y(x_1,x_2)$ for $(x_1,x_2)\in D$, where the constant $k_2$ will again be adjusted later. We first note that
$y_{x_2}(x_1,0)>0$, $\forall x_1 \in\R$, since the function $y$ vanishes at $(x_1,0)$, is positive in $\{x_2> 0\}$, and satisfies \eqref{painhom}. This and  $\inf_{x_2\geq d}  y(1,x_2)>0$, $\forall d>0$, imply that when $k_2$ is large enough, we have $\psi(1,x_2)\leq 0$, $\forall x_2\geq 0$. On the other hand, it is clear that $\psi(x_1,0)<0$, $\forall x_1\geq 1$.
 Next, we compute
\begin{align*}
\Delta \psi&=(x_1+6y^2)(-y_{x_2})+(x_1+2y^2)(-k_2y)\geq (x_1+6y^2)\psi.
\end{align*}
Thus, by applying Lemma \ref{maxpp}, it follows that $\psi\leq 0$ in $D$.

Finally, setting $\frac{\nabla y}{|\nabla y|}=e^{i\phi}$ when $(x_1,x_2)\in D$ (with $\phi \in (\frac{\pi}{2},\frac{3\pi}{2})$), we find that 
\[
\tan \phi=\frac{y_{x_2}}{y_{x_1}}\leq \frac{k_1k_2}{\sqrt{x_1}}\Longrightarrow \phi \leq \pi+\arctan \Big(\frac{k_1k_2}{\sqrt{x_1}}\Big).
\]
As a consequence, we have $\nabla y(x_1,x_2)\cdot \n>0$ if $\theta \in \big(\arctan \big(\frac{k_1k_2}{\sqrt{x_1}}\big),\frac{\pi}{2}\big)$, that is, if $x_1>s_\n:=\big(\frac{k_1k_2}{\tan \theta}\big)^2$. 
\end{proof}

\begin{lemma}\label{movingplane}
Let $\theta\in (0,\frac{\pi}{2})$ be fixed, and consider for every $\lambda \in \R$ the reflection $\sigma_\lambda$ with respect to the line $\Gamma_\lambda:=\{(x_1,x_2): x_2=\tan\theta (x_1-\lambda)\}$, and the domain 
$D_\lambda:=\{(x_1,x_2): 0<x_2<\tan\theta (x_1-\lambda)\}$.
Then, the function $\psi_\lambda(x_1,x_2):=y(x_1,x_2)-y(\sigma_\lambda(x_1,x_2))$ 
is negative in $D_\lambda$, for every $\lambda\in\R$. 
\end{lemma}
\begin{proof}
We set $\n=e^{i (\theta+\frac{\pi}{2})}$ as in Lemma \ref{lll2}, and denote by $(p',q')$ the image by $\sigma_\lambda$ of a point $(p,q)\in D_\lambda$, and by $D'_\lambda$ the set $\sigma_\lambda(D_\lambda)$. 
It is obvious that $\psi_\lambda(x_1,0)< 0$, $\forall x_1> \lambda$, and that $\psi_\lambda(x_1,x_2)= 0$, $\forall (x_1,x_2)\in \Gamma_\lambda$. Moreover, $\psi_\lambda$ satisfies 
$$\Delta \psi_\lambda(p,q)-c(p,q)\psi_\lambda= (p-p')y(p',q')\geq 0, \quad  \forall (p,q)\in D_\lambda,$$ with $c(p,q)=p+2(y^2(p,q)+y(p,q)y(p',q')+y^2(p',q'))$.
For each $\lambda\in\R$ we consider the statement
\begin{equation}\label{statem}
\psi_\lambda(p,q) <0, \quad \forall (p,q)\in D_\lambda.
\end{equation}

We shall first establish Lemma \ref{movingplane} in the case where $\theta\in (0,\frac{\pi}{4})$.
According to Lemma \ref{lll2}, \eqref{statem} is valid for each $\lambda\geq s_\n$.
Set $\lambda_0=\inf\{\lambda\in \R: \psi_\mu<0 \text{ holds in $ D_\mu$, for each $\mu\geq \lambda$} \}$. We will prove $\lambda_0=-\infty$.
Assume instead $\lambda_0\in\R$. Then, there exist a sequence $\lambda_k<\lambda_0$ such that $\lim_{k\to\infty}\lambda_k=\lambda_0$, and a sequence
$(p_k,q_k)\in D_{\lambda_k}$, such that $y(p_k,q_k)\geq y(p'_k,q'_k)$. According to Lemma \ref{lll2}, we have $p'_k\leq s_\n$, thus the sequence $(p_k,q_k)$ is bounded.
Up to subsequence we may assume that $\lim_{k\to\infty}(p_k,q_k)=(p_0,q_0)\in\overline{D_{\lambda_0}}$, with $p'_0\leq s_\n$.
By definition of $\lambda_0$, we have $\psi_{\lambda_0}\leq 0$ in $D_{\lambda_0}$, and $\psi_{\lambda_0}(p_0,q_0)=0$ i.e. $y(p_0,q_0)=y(p'_0,q'_0)$.
Now we distinguish the following cases. If $(p_0,q_0)\in D_{\lambda_0}$, the maximum principle implies that $\psi_{\lambda_0}\equiv 0$ in $D_{\lambda_0}$. Clearly, this situation is excluded, since $y$ is positive
in the half-plane $\{x_2>0\}$.
On the other hand, the maximum principle also implies that $\frac{\partial \psi_{\lambda_0}}{\partial \n}(p,q)=2\frac{\partial y}{\partial \n}(p,q)>0$, provided that $(p,q)\in \Gamma_{\lambda_0}$ and $q>0$.
Furthermore, the previous inequality still holds at the vertex $(p,q)=(\lambda_0,0)$, since $y_{x_2}(x_1,0)>0$ and $y_{x_1}(x_1,0)=0$, $\forall x_1\in\R$ (cf. the proof of Lemma \ref{lll2}).
As a consequence, in a neighborhood of the line segment $\{(x_1,x_2): x_2=\tan\theta (x_1-\lambda), 0\leq x_1\leq s_\n\}$, we have that $\frac{\partial y}{\partial \n}>0$, and it follows that $(p_0,q_0)$
cannot belong to $\Gamma_{\lambda_0}$. Finally, since the case where $p_0>\lambda_0$ and $q_0=0$ is ruled out (because $y$ is positive in the half-plane $\{x_2>0\}$),
we have reached a contradiction.
\begin{figure}[h]
\includegraphics{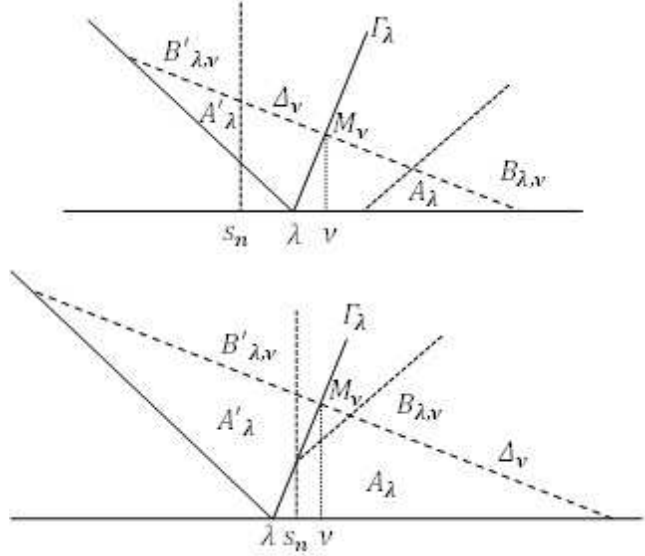}
\caption{The sets $A_\lambda$, $A'_\lambda$, $B_{\lambda,\nu}$, $B'_{\lambda,\nu}$, and the lines $\Gamma_\lambda$, $\Delta_\nu$, in the case where $\lambda>s_\n$ and $\lambda<s_\n$. }
\label{fig}
\end{figure}

Next, we establish Lemma \ref{movingplane} in the case where $\theta \in [\frac{\pi}{4},\frac{\pi}{2})$, which is a little bit more involved. When $\theta=\frac{\pi}{4}$, it is clear that
\eqref{statem} is valid for each $\lambda\geq s_\n$. Otherwise, when $\theta \in (\frac{\pi}{4},\frac{\pi}{2})$, let $A'_\lambda:=\{(p',q')\in D'_\lambda: p'\leq s_\n\}$, and let $A_\lambda=\sigma_\lambda(A'_\lambda)$. Our first claim is that $m:=\inf_{A'_{s_\n+1}}y>0$. Indeed, proceeding as in the proof of Theorem \ref{corpain} (v), one can see that 
\[
\lim_{(x_1,x_2)\in A'_{x_\n+1},x_1\to-\infty}\frac{\sqrt{2}}{\sqrt{-x_1}}y(x_1,x_2)=1.
\] 
In addition, proceeding as in the proof of Lemma \ref{lll2}, we obtain that $\inf\{y(x_1,x_2): (x_1,x_2)\in A'_{s_\n+1}, \, s_\n-l\leq x_1\leq s_\n\}>0$, for every constant $l>0$. Thus, $m>0$. On the other hand, we have $\lim_{\lambda\to\infty}\sup\{ y(x_1,x_2): (x_1,x_2)\in A_\lambda\}=0$, since $\lim_{\lambda\to\infty}\inf\{ x_1: (x_1,x_2)\in A_\lambda\}=0$ (cf. Lemma \ref{expcvv}). As a consequence when $\lambda\geq s_\n+1$ is large enough, we have $y(p',q')\geq m>y(p,q)$, $\forall (p,q)\in A_\lambda$, and also $y(p',q')>y(p,q)$, $\forall (p,q)\in D_\lambda\setminus A_\lambda$, by definition of $s_\n$. This establishes that \eqref{statem} holds for $\lambda$ large enough.
Then, defining $\lambda_0$ as previously, we assume by contradiction that $\lambda_0\in\R$, and deduce in a similar way the existence of the sequences $\lambda_k$ and $(p_k,q_k)\in D_{\lambda_k}$.
We need to show that $(p_k,q_k)$ is bounded. For $\nu>\lambda$, let $M_\nu:=(\nu,\tan\theta(\nu-\lambda))\in\Gamma_\lambda$, and let $\Delta_\nu:=\{(x_1,x_2): x_2=\tan(\theta+\frac{\pi}{2})(x_1-\nu)+\tan \theta (\nu-\lambda)\}$ be the line parallel to $\n$ and passing through $M_\nu$. Let also $B'_{\lambda,\nu}:=\{(p',q')\in A'_\lambda: q'\geq\tan(\theta+\frac{\pi}{2})(p'-\nu)+\tan \theta (\nu-\lambda)\}$ be the subset of $A'_\lambda$ which is above $\Delta_\nu$, and $B_{\lambda,\nu}:=\sigma_\lambda(B'_{\lambda,\nu})$. Proceeding as previously, we can see that $\forall \nu>\lambda_0+2$, $\forall \lambda >\lambda_0-1$, we have $\inf_{B'_{\lambda,\nu}}y>m$ for some constant $m>0$, while $\lim_{\nu\to\infty}\sup \{y(x_1,x_2): (x_1,x_2)\in B_{\lambda,\nu}\}=0$.
As a consequence, for $\nu$ large enough and $\lambda>\lambda_0-1$, we have $y(p',q')\geq m> y(p,q)$, $\forall (p,q)\in B_{\lambda,\nu}$, and thus $(p_k,q_k)\notin B_{\lambda_k,\nu}$. Furthermore, since $p'_k\leq s_\n$ by Lemma \ref{lll2}, we have established the boundedness of $(p_k,q_k)$. To complete the proof we utilize the same arguments detailed in the case where $\theta\in (0,\frac{\pi}{4})$.

\end{proof}

Lemma \ref{movingplane} implies that $\forall \theta \in (0,\frac{\pi}{2})$, $\forall \lambda \in \R$, and $(p,q)\in \Gamma_{\lambda}$ with $q>0$, we have
$\frac{\partial \psi_{\lambda}}{\partial \n}(p,q)=2\frac{\partial y}{\partial \n}(p,q)>0$, where $\n=e^{i(\theta+\frac{\pi}{2})}$.
It follows that $y_{x_1}(x_1,x_2)\leq 0$, and $y_{x_2}(x_1,x_2)\geq 0$, $\forall x_1\in\R$, $\forall x_2\geq 0$.
Moreover, in the half-plane $x_2\geq 0$, $y_{x_1}$ and $y_{x_2}$ satisfy respectively $\Delta y_{x_1}\geq (x_1+6y^2)y_{x_1}$, and $\Delta y_{x_2}= (x_1+6y^2)y_{x_2}$, thus $y_{x_1}$ (resp. $y_{x_2}$) cannot vanish in the open half-plane $x_2>0$, since otherwise we would obtain by the maximum principle $y_{x_1}\equiv 0$ (resp. $y_{x_2}\equiv 0$). These situations are excluded by the fact that $y>0$ in the open half-plane $x_2>0$, and $y_{x_2}(x_1,0)>0$, $\forall x_1\in\R$. Therefore we have proved that $y_{x_1}(x_1,x_2)< 0$, $\forall x_1\in\R$, $\forall x_2>0$, and $y_{x_2}(x_1,x_2)> 0$, $\forall x_1,x_2\in\R$. Finally, setting $\tilde y_l(x_1,x_2)=y(x_1,x_2+l)$, we obtain by the Theorem of Ascoli, that up to a subsequence $l_k\to\infty$ , $\tilde y_{l_k}$ converges in $C^2_{\mathrm{loc}}$ to a nonnegative minimal solution $\tilde y_\infty$ of \eqref{painhom}. Furthermore, the monotonicity of $y$ along the $x_2$ direction implies that $\tilde y_\infty$ is independent of $x_2$. Thus, since $h$ is the only nonnegative minimal solution of \eqref{painhom} (cf. \cite[Theorem 1.3]{Clerc2017}), we deduce that $\tilde y_\infty(x_1,x_2)=h(x_1)$, and that $\lim_{l\to\infty} y(x_1,x_2+l)=h(x_1)$ is independent of the subsequence $l_k$. We also note that $|y(x_1,x_2)|<h(x_1)$, $\forall (x_1,x_2)\in \R^2$, from which Theorem \ref{corpain} (iv) follows.
This completes the proof of Theorem \ref{corpain}.
\end{proof}

\end{document}